\documentclass[12pt,a4paper, final, twoside]{article}
\usepackage{amsmath}
\usepackage{amsthm}
\usepackage{amsfonts}
\usepackage{amssymb}
\usepackage{amscd}
\usepackage{latexsym}
\usepackage{amsthm}
\usepackage{graphicx}
\usepackage{graphics}
\usepackage[colorlinks=true, urlcolor=blue,  linkcolor=black, citecolor=black]{hyperref}
\usepackage{color}
\usepackage{dsfont}
\newtheorem{thm}{Theorem}[section]

\newtheorem{lem}[thm]{Lemma}
\newtheorem{example}[thm]{Example}

\theoremstyle{notice}
\newtheorem{notice}[thm]{Notice}

\theoremstyle{proposition}
\newtheorem{prop}{Proposition}[section]

\newtheorem{cor}[thm]{Corollary}

\theoremstyle{definition}
\newtheorem{defn}{Definition}[section]
\newtheorem{rem}{Remark}[section]

\numberwithin{equation}{section}
\begin{document}
\hyphenpenalty=100000

\begin{center}

{\Large \textbf{\\On Topological Structure of the First Non-abelian Cohomology of Topological Groups }}\\[5mm]
{\large \textbf{H. Sahleh$^\mathrm{1}$\footnote{\emph{$^\mathrm{1}$E-mail: sahleh@guilan.ac.ir $^\mathrm{2}$E-mail: h.e.koshkoshi@guilan.ac.ir}}  and H.E. Koshkoshi$^\mathrm{2}$}}\\[1mm]
$^\mathrm{1}${\footnotesize \it Department of  Mathematics, Faculty of Mathematical Sciences, University of Guilan,\\ P. O. Box 1914, Rasht, Iran}\\[1mm]
$^\mathrm{2}${\footnotesize \it Department of  Mathematics, Faculty of Mathematical Sciences, University of Guilan}
\end{center}

\begin{center}
\textbf{Abstract}
\end{center}
\par Let $G$, $R$, and $A$ be topological groups. Suppose that $G$ and $R$ act continuously on $A$, and $G$ acts continuously on $R$. In this paper, we define a partially crossed topological $G-R$-bimodule $(A,\mu)$, where $\mu:A\rightarrow R$ is a continuous homomorphism. Let $Der_{c}(G,(A,\mu))$ be the set of all $(\alpha,r)$ such that $\alpha:G\rightarrow A$ is a continuous crossed homomorphism and  $\mu\alpha(g)=r^{g}r^{-1}$. We introduce a topology on $Der_{c}(G,(A,\mu))$. We show that $Der_{c}(G,(A,\mu))$ is a topological group, wherever $G$ and $R$ are locally compact. We define the first cohomology, $H^{1}(G,(A,\mu))$, of $G$ with coefficients in $(A,\mu)$ as a quotient space of $Der_{c}(G,(A,\mu))$. Also, we state conditions under which $H^{1}(G,(A,\mu))$ is a topological group. Finally, we show that under what conditions $H^{1}(G,(A,\mu))$ is one of the following:  $k$-space, discrete, locally compact and  compact.
\\[1mm]

\section{Introduction}\label{section 1}
The first non-abelian cohomology of groups with coefficients in crossed modules (algebraically)  was introduced by Guin \cite{4'}. The Guin's approach is extended by Inassaridze to any dimension with coefficients in (partially) crossed bimodules (\cite{6},\cite{7}). Hu \cite{6'} defined the  cohomology of topological groups  with coefficients in abelian topological modules. This paper is a part of an investigation about \emph{non-abelian cohomology of topological groups}.  We consider the first non-abelian cohomology in the topological context.  The methods used here are motivated by \cite{6} and \cite{7}.
\par All topological groups are assumed to be Hausdorff (not necessarily abelian), unless otherwise specified.  Let $G$  and $A$ be topological groups. It is said that $A$ is a topological $G$-module, whenever $G$   acts continuously on the  left of $A$. For all $g\in G$ and $a\in A$ we denote the action of $g$ on $a$ by $^{g}a$. The centre and the commutator of a topological group $G$ is denoted by $Z(G)$ and $[G,G]$, respectively. If $G$ and $H$ are topological groups and $f:G\to H$ is a continuous homomorphism we denote by $\bar{f}:G\to f(G)$ the restricted map of $f$ on its range and by $\mathbf{1}:G\rightarrow H$ the trivial homomorphism. The topological isomorphism and isomorphism are denoted respectively by $"\simeq" $ and $"\cong" $. If the topological groups $G$ and $R$ act continuously on a topological group $A$, then the notation $^{gr}a$ means $^{g}(^{r}a)$, $g\in G$, $r\in R$, $a\in A$. We assume that every topological group acts on itself by conjugation.
\par In section \ref{section 2}, we define precrossed, partially crossed and crossed topological $R$-module $(A,\mu)$, where $A$ is a topological $R$-module and $\mu:A\rightarrow R$ is a continuous homomorphism. Also, we generalize, these definitions to precrossed, partially crossed and crossed topological $G-R$-bimodule $(A,\mu)$, when $G$ and $R$ act continuously on $A$, and $G$ acts continuously on $R$. We define the set $Der_{c}(G,(A,\mu))$, for a partially crossed  topological $G-R$-bimodule. We denote the set of all continuous maps from $G$ into $A$, with compact-open topology, by $\mathcal{C}_{k}(G,A)$. Since $Der_{c}(G,(A,\mu))\subset \mathcal{C}_{k}(G,A)\times R$, then we may consider  $Der_{c}(G,(A,\mu))$ as a topological subspace of $\mathcal{C}_{k}(G,A)\times R$. We show that $Der_{c}(G,(A,\mu))$ is a topological group, whenever $G$ and $R$ are locally compact (Theorem \ref{thm 2.16}). In addition, we prove that $Der_{c}(G,(A,\mu))$ is a topological $G$-module. Furthermore, we show that under what conditions,  $Der_{c}(G,(A,\mu))$ is a precrossed topological $G-R$- bimodule (Proposition \ref{prop 2.23}).
\par In section \ref{section 3}, we define $H^{1}(G,(A,\mu))$ as a quotient of $Der_{c}(G,(A,\mu))$, where $(A,\mu)$ is a partially crossed topological $G-R$-bimodule. We state conditions under which $H^{1}(G,(A,\mu))$ is a topological group (see Theorem \ref{thm 3.2}).  Moreover, since each partially crossed topological $G$-module{can be naturally viewed} as a partially crossed topological $G-G$-bimodule, then we may define $H^{1}(G,(A,\mu))$, when $(A,\mu)$ is a partially crossed topological $G$-module. Finally, we find conditions under which $H^{1}(G,(A,\mu))$ is one of the following:  $k$-space, discrete, locally compact and  compact.

\section{Partially Crossed topological $G-R$-bimodule $(A,\mu)$}\label{section 2}
 In this section, we define a partially crossed topological $G-R$-bimodule $(A,\mu)$. We give some examples of precrossed, {partially} crossed and crossed topological $G-R$-bimodules. Also, {we} define $Der_{c}(G,(A,\mu))$ and prove that if $G$ and $R$ are locally compact, then  $Der_{c}(G,(A,\mu))$ is a  topological group. Moreover, if the topological groups $G$ and $R$ act continuously on each other and on $A$ compatibly, then $(Der_{c}(G,(A,\mu)),\gamma)$ is a precrossed topological $G-R$-bimodule, where $\gamma:Der_{c}(G,(A,\mu))\rightarrow R$, $(\alpha,r)\mapsto r$.
\begin{defn} By a precrossed topological $R$-module we mean a pair $(A,\mu)$  where $A$ is a topological $R$-module and  $\mu:A\to R$ is a continuous homomorphism such that
$$\mu(^{r}a)=\ ^{r}\mu(a), \forall r\in R, a\in A.$$
\par If in addition we have the \emph{Pieffer identity}
\begin{center}
$^{\mu(a)}b=\ ^{a}b$, $\forall a, b\in A,$
\end{center}
 then $(A,\mu)$ is called a crossed topological $R$-module.
 \label{def 2.1}\end{defn}
 \begin{defn} A precrossed topological $R$-module $(A,\mu)$ is said to be a partially crossed topological $R$-module, whenever it satisfies the following equality
$$^{\mu(a)}b=\ ^{a}b,$$ for all $b\in A$ and for all $a\in A$ such that $\mu(a)\in [R,R]$.
\label{def 2.2}\end{defn}
 It is clear that every crossed topological $R$-module is  a partially crossed topological $R$-module.
\begin{example} Suppose that $A$ is a non-abelian topological group with nilpotency class of two (i.e., $[A,A]\subseteq Z(A)$). Take $R=A/\overline{[A,A]}$. Let $\pi:A\rightarrow R$ be the canonical surjective map and suppose that $R$ acts trivially on $A$. It is clear that $^{\pi(a)}b=\ ^{a}b,$ for all $b\in A$ if and only if $a\in Z(A)$. Hence, $(A,\pi)$ is a partially crossed topological $R$-module which is not a crossed topological $R$-module.
\label{example 2.3}\end{example}
\begin{defn} Let $G$, $R$ and $A$ be topological groups. A precrossed topological $R$-module $(A,\mu)$ is said to be  a precrossed topological $G-R$-bimodule, whenever
\begin{itemize}
\item[(1)] $G$  acts continuously on $R$ and $A$;
\item[(2)] $\mu:A\rightarrow R$ is a continuous $G$-homomorphism;
\item[(3)] $^{(^gr)}a=\ ^{grg^{-1}}a$ (i.e., compatibility condition) for all $g\in G$, $r\in R$ and $a\in A$.
\end{itemize}
\label{def 2.4}\end{defn}
 \begin{defn} A precrossed topological $G-R$-bimodule $(A,\mu)$ is said to be a crossed topological $G-R$-bimodule, if  $(A,\mu)$ is a crossed topological $R$-module.
\label{def 2.5}\end{defn}
\begin{example} (1) Let $A$ be an arbitrary topological G-module. Then $Z(A)$ is a topological $G$-module. Since $A$ is Hausdorff, then $Z(A)$ is a closed subgroup of $A$. Thus, the quotient group $R=A/Z(A)$ is  Hausdorff. Now, we define an action of $R$ on $A$ and an
action of $G$ on $R$ by:
$$
 ^{aZ(A)}b=\ ^{a}b, \forall a,b\in A, \ \ \ \ ^{g}(aZ(A))=\ ^{g}a, \forall g\in G, a\in A. \eqno{(2.1)}
$$
Let $\pi_{A}:A\rightarrow R$  be the canonical homomorphism. It is easy to see that under $(2.1)$ the pair
$(A,\pi_{A})$ is a  crossed topological $G-R$-bimodule.
\begin{itemize}
  \item[(2)] By part (1), for any topological group $G$ the pair $(G,\pi_{G})$ is a crossed topological $G-G/Z(G)$-bimodule.
\end{itemize}
\label{example 2.6}\end{example}
\begin{defn} A precrossed topological $G-R$-bimodule  $(A,\mu)$ is said to be partially crossed topological $G-R$-bimodule, if  $(A,\mu)$ is a partially crossed topological $R$-module.
\label{def 2.7}\end{defn}
Let $G$ be a locally compact group and $Aut(G)$ the group of
all topological group automorphisms (i.e., continuous and open automorphisms) of $G$ with the \emph{Birkhoff topology} (see \cite{2}, \cite{3} and \cite{5}). This topology is known as the \emph{generalized compact-open topology}. A
neighborhood basis of the identity automorphism consists of sets
$N(C,V) = \{\alpha \in Aut(G): \alpha(x) \in Vx, \alpha^{-1}(x) \in Vx, \forall x \in C\}$, where $C$
is a compact subset of $G$ and $V$ is a neighborhood of the identity of $G$. It is
well-known that $Aut(G)$ is a Hausdorff topological group (see page 40 of \cite{5}). The generalized compact-open topology is finer than the compact-open topology in $Aut(G)$ and if $G$ is compact, then the generalized compact-open topology coincides with compact-open topology in $Aut(G)$ (see page 324 of \cite{3}).
\begin{lem} Let $A$ be a locally compact group and $G$ a topological group. Suppose that $A$ is a topological $G$-module. Then
\begin{itemize}
  \item[(i)] the homomorphism $\imath_{A}:A\rightarrow Aut(A)$, $a\mapsto c_{a}$, is continuous, where $c_{a}(b)=aba^{-1}, \forall b\in A$;
  \item[(ii)] $A$ is a topological $Aut(A)$-module by the action $^{\alpha}x=\alpha(x)$, $\forall \alpha \in Aut(A), x\in A$;
  \item[(iii)] $Aut(A)$ is a topological $G$-module by the action $(^{g}\alpha)(x)=\ ^{g}\alpha(^{g^{-1}}x)$, $\forall g\in G, \alpha \in Aut(A), x\in A$.
\end{itemize}
\label{lem 2.8}\end{lem}
\begin{proof} For (i) and (ii) see  page 324 of \cite{3}, and Proposition 3.1 of \cite{5}. (iii):  It is enough to prove that the map $\chi:G\times Aut(A)\rightarrow Aut(A)$, $(g,\alpha)\mapsto \ ^{g}\alpha$ is continuous.  By (ii), the maps $\phi:(G\times Aut(A))\times A\rightarrow A$, $((g,\alpha),x)\mapsto \ ^{g}\alpha(^{g^{-1}}x)x^{-1}$ and $\psi:(G\times Aut(A))\times A\rightarrow A$, $((g,\alpha),x)\mapsto \ ^{g}\alpha^{-1}(^{g^{-1}}x)x^{-1}$ are continuous. Let $^{g}\alpha\in N(C,V)$. Then, $\phi((g,\alpha),x)\in V$ and $\psi((g,\alpha),x)\in V$, for all $x\in C$. Thus, $\phi(\{(g,\alpha)\}\times C)\subset V$ and $\psi(\{(g,\alpha)\}\times C)\subset V$. Now,  $\phi^{-1}(V)$ and $\psi^{-1}(V)$ are open in $(G\times Aut(A))\times A$ containing $\{(g,\alpha)\}\times C$. Hence, $\phi^{-1}(V)\cap \psi^{-1}(V)\cap (G\times Aut(A))\times C$ is an open set in $ (G\times Aut(A))\times C$ containing the slice $\{(g,\alpha)\}\times C$ of $ (G\times Aut(A))\times C$. The tube lemma (Lemma 5.8 of \cite{11}) implies that there is an open neighbourhood $U$ of $(g,\alpha)$ in $G\times Aut(A)$ such that the tube $U\times C$ lies in $\phi^{-1}(V)\cap \psi^{-1}(V)$. Then, for every $(h,\beta)\in U$,  $x\in C$, we have $\phi((h,\beta),x)\in V$ and $\psi((h,\beta),x)\in V$, i.e., $^{h}\beta(^{h^{-1}}x)\in Vx$ and $^{h}\beta^{-1}(^{h^{-1}}x)\in Vx$. Therefore, $^{h}\beta\in N(C,V)$, for all $(h,\beta)\in U$. So $\chi$ is continuous.
\end{proof}
\begin{prop} Let $A$ be a topological $G$-module and $A$ a locally compact group. Then, $(A,\imath_{A})$ is a crossed topological $G-Aut(A)$-bimodule, where the homomorphism $\imath_{A}$ and the actions are defined as in Lemma \ref{lem 2.8}.
\label{prop 2.9}\end{prop}
\begin{proof}  By Lemma \ref{lem 2.8}, the homomorphism $\imath_{A}$ and the actions are continuous. Also,
 \par 1. For every $g\in G$ and $a,b\in A$, $\imath_{A}(^{g}a)(b)=c_{^{g}a}(b)=\ ^{g}ab^{g}a^{-1}=\ ^{g}c_{a}(b)$. Hence, $\imath_{A}$ is a $G$-homomorphism.
 \par 2. For every $\alpha\in Aut(A)$  and $x, a\in A$,{} $\imath_{A}(^{\alpha}x)(a)=\imath_{A}(\alpha(x))(a)=c_{\alpha(x)}(a)=\alpha(x)a\alpha(x)^{-1}=\alpha(x\alpha^{-1}(a)x^{-1})=\alpha\circ c_{x}\circ \alpha^{-1}(a)=\ ^{\alpha}c_{x}(a)$. So $\imath_{A}$ is a $Aut(A)$-homomorphism.
 \par 3. For every $a, b\in A$, {} $^{\imath_{A}(a)}b=c_{a}(b)=aba^{-1}=\ ^{a}b$. Thus, the Pieffer identity is satisfied.
  \par 4. The compatibility condition is satisfied. Since for every $g\in G, \alpha\in Aut(A), x\in A$, then $^{^{g}\alpha}x=(^{g}\alpha)(x)=\ ^{g}\alpha(^{g^{-1}}x)=\ ^{g\alpha g^{-1}}x$.
   \par Therefore, $(A,\imath_{A})$ is a crossed topological $G-Aut(A)$-bimodule.
\end{proof}
\begin{rem} In a natural way any precrossed (crossed) topological $R$-module is a precrossed (crossed) topological $R-R$-bimodule.
\label{rem 2.10}\end{rem}
\begin{rem} Let $(A,\mu)$ be a partially crossed (crossed) topological $G-R$-bimodule. Then,  $(A,\overline{\mu})$ is a partially crossed (crossed) topological $G-\mu(A)$-bimodule. Thus, by Proposition \ref{prop 2.9}, for any topological $G$-module $A$ in which $A$ is locally compact, we may associate the crossed topological $G-Inn(A)$-bimodule $(A,\overline{\imath_{A}})$, where $Inn(A)$ is the topological group of all inner automorphisms of $A$.
\label{rem 2.11}\end{rem}
\begin{defn} Let $(A,\mu)$ be a partially crossed topological $G-R$-bimodule. The map $\alpha:G\rightarrow A$ is called a crossed homomorphism whenever, $$\alpha(gh)=\alpha(g)^{g}\alpha(h), \forall g, h\in G.$$ Denote by $Der(G,(A,\mu))$ the set of all pairs $(\alpha,r)$ where $\alpha:G\rightarrow A$ is a crossed homomorphism and $r$ is an element of $R$ such that $$\mu\circ \alpha(g)=\ r^{g}r^{-1}, \forall g\in G.$$
\label{def 2.12}\end{defn}
Let $Der_{c}(G,(A,\mu))=\{(\alpha,r) | (\alpha,r)\in Der(G,(A,\mu))$ and $\alpha$ is continuous$\}$. H. Inassaridze \cite{7} introduced the product $\star$ in $Der(G,(A,\mu))$  by
\begin{center}
$(\alpha,r)\star(\beta,s)=(\alpha*\beta,rs)$, where $\alpha*\beta(g)=\ ^{r}\beta(g)\alpha(g), \forall g\in G.$
\end{center}
\begin{defn} A family $\eta$ of subsets of a topological space
$X$ is called a network on $X$ if for each point $x\in X$ and
each neighbourhood $U$ of $x$ there exists $P\in \eta$ such that
$x \in P\subset U$. A network $\eta$ is said to be compact
(closed) if all its elements are compact (closed) subspaces of
$X$. We say that a closed network $\eta$ is  hereditarily closed
if for each $P\in \eta$ and any closed set $B$ in  $P$,
{} $B\in \eta$. \label{def 2.13}\end{defn}
Let $X$ and $Y$ be topological spaces. The set of all continuous functions $f:X\rightarrow Y$ is denoted by $\mathcal{C}(X,Y)$. Suppose that $U\subset X$ and $V\subset Y$. Take
$$[U, V] = \{f \in \mathcal{C}(X,Y) : f(U) \subset V\}.$$
Let $X$ and $Y$ be topological spaces, and $\eta$ a network in $X$. The family
$\{[P, V] : P \in \eta$ and $V$ is open in $Y\}$ is a subbase {for} a topology on $\mathcal{C}(X,Y)$,
called the $\eta$-topology. We denote the set $\mathcal{C}(X,Y)$ with the $\eta$-topology by $\mathcal{C}_{\eta}(X,Y)$. If $\eta$ is the family of all singleton subsets of $X$, then the $\eta$-topology is called the point-open topology; in this case $\mathcal{C}_{\eta}(X,Y)$ is
denoted by $\mathcal{C}_{p}(X,Y)$. If $\eta$ is the family of all compact subspaces of $X$, then the
$\eta$-topology is called the compact-open topology and $\mathcal{C}_{\eta}(X,Y)$ is denoted by
$\mathcal{C}_{k}(X,Y)$ (see \cite{9}).
\par Now, suppose that $A$ is a topological group, then $\mathcal{C}(X, A)$ is a
group. For $f, g\in \mathcal{C}(X, A)$ the product, $f.g$, is defined by
$$(f.g)(x) = f(x).g(x), \forall x \in X.  \eqno{(2.2)}$$
\begin{lem} Let $X$ be a Tychonoff space and $A$ a topological group. If $\eta$ is a hereditarily closed, compact network on $X$, then  under  the product $(2.2)$, $\mathcal{C}_{\eta}(X,A)$ is a topological group. In particular, $\mathcal{C}_{p}(X,A)$ and $\mathcal{C}_{k}(X,A)$ are topological groups.
\label{lem 2.14}\end{lem}
\begin{proof} See Theorem 1.1.7 of \cite{9}. In particular, the set of all finite {subsets} of $X$ and the set of all compact {subsets} of $X$ are hereditarily  closed, compact networks on $X$.
\end{proof}
 Suppose that $X$ is a topological space and $A$ a topological $R$-module. Then, $\mathcal{C}(X, A)$ is an $R$-module. If $r\in R, f\in \mathcal{C}(X, A)$, then the action $^{r}f$ is defined by
$$(^{r}f)(x) =\ ^{r}(f(x)), \forall x \in X. \eqno{(2.3)}$$
\begin{prop} Let $X$ be a locally compact Hausdorff space, $R$ a locally compact group and $A$ a topological $R$-module. Then, by (2.3), $\mathcal{C}_{k}(X,A)$ is a topological $R$-module.
\label{prop 2.15}\end{prop}
\begin{proof} Since $X$ is a locally compact Hausdorff space, then by Lemma \ref{lem 2.14}, $\mathcal{C}_{k}(X,A)$ is a topological group. By Theorem 5.3 of \cite{11}, the evaluation map $e:X\times \mathcal{C}_{k}(X,A)\rightarrow A$, $(x,f)\mapsto f(x)$ is continuous. Thus, the map $F:R\times X\times \mathcal{C}_{k}(X,A) \rightarrow A$, $(r,x,f)\mapsto \ ^{r}f(x)$ is continuous. By Corollary 5.4 of \cite{11},  the induced map $\hat{F}:\mathcal{C}_{k}(X,A) \rightarrow \mathcal{C}_{k}(R\times X,A)$  is continuous, where $\hat{F}$ is defined by $$\hat{F}(f)(r,x)=\ ^{r}f(x).$$
On the other hand the exponential
map $\Lambda: \mathcal{C}_{k}(R\times X,A)\rightarrow \mathcal{C}_{k}(R,C_{k}(X,A))$, $u\mapsto \Lambda(u)$;  $\Lambda(u)(r)(x)=u(r,x)$, is a homeomorphism (see Corollary 2.5.7 of \cite{9}. Therefore, $\Lambda \circ \hat{F}:\mathcal{C}_{k}(X,A)\rightarrow \mathcal{C}_{k}(R,\mathcal{C}_{k}(X,A))$ is a continuous map. Since $R$ is locally compact and Hausdorff then by Corollary 5.4 of \cite{11}, $\Lambda \circ \hat{F}$  induces  the continuous map $\chi:R\times \mathcal{C}_{k}(X,A)\rightarrow \mathcal{C}_{k}(X,A)$, $\chi(r,f)=(\Lambda \circ \hat{F}(f))(r)=\ ^{r}f$.
 Therefore,  $\mathcal{C}_{k}(X,A)$ is a topological $R$-module.
\end{proof}
Note that $Der_{c}(G,(A,\mu))\subset Der_{c}(G,A)\times R\subset \mathcal{C}(G,A)\times R$, where $Der_{c}(G,A)=\{\alpha| \alpha$ is a continuous crossed homomorphism from $G$ into $A$$\}$. Thus, $\mathcal{C}_{k}(G,A)\times R$ induces the subspace topology on  $Der_{c}(G,(A,\mu))$. Here,  the induced subspace topology on  $Der_{c}(G,(A,\mu))$ is called the \emph{induced topology by compact-open topology}. From now on, we consider $Der_{c}(G,(A,\mu))$ with this topology.
\begin{thm} Let $G$ and $R$ be  locally compact groups and $(A,\mu)$ a partially crossed topological $G-R$-bimodule.  Then, $(Der_{c}(G,(A,\mu)),\star)$  is a topological group.
\label{thm 2.16}\end{thm}
\begin{proof} By Proposition 3 of \cite{7}, $Der(G,(A,\mu))$ is a group. If $(\alpha,r), (\beta,s)\in Der_{c}(G,(A,\mu))\subset Der(G,(A,\mu))$, then $(\alpha,r)\star (\beta,s)\in Der_{c}(G,(A,\mu))$ and $(\alpha,r)^{-1}=(\bar{\alpha},r^{-1}) \in Der_{c}(G,(A,\mu))$, where $\bar{\alpha}(g)=\ ^{r^{-1}}\alpha(g)^{-1}, \forall g\in G $. It is clear that $\alpha*\beta$ and $\bar{\alpha}$ are continuous. Thus, $Der_{c}(G,(A,\mu))$ is a subgroup of $Der(G,(A,\mu))$.
\par By Proposition \ref{prop 2.15}, $\mathcal{C}_{k}(G,A)$ is a topological $R$-module. Thus, it is clear that  $$\phi:(\mathcal{C}_{k}(G,A)\times R)\times (\mathcal{C}_{k}(G,A)\times R)\rightarrow \mathcal{C}_{k}(G,A)\times R$$ $$ ((f,r),(g,s))\mapsto (^{r}gf,rs)$$ and  $$\psi:\mathcal{C}_{k}(G,A)\times R\rightarrow \mathcal{C}_{k}(G,A)\times R$$ $$ (f,r)\mapsto \bar{f}=(^{r^{-1}}f^{-1},r^{-1})$$ are continuous. Obviously,  the restrictions of $\phi$ and $\psi$ to $Der_{c}(G,(A,\mu))\times Der_{c}(G,(A,\mu))$ and $Der_{c}(G,(A,\mu))$ are continuous, respectively. Consequently, $(Der_{c}(G,(A,\mu)),\star)$  is a topological group.
\end{proof}
\begin{prop} (i) Let $(A,\mu)$ be a partially crossed topological $G-R$-bimodule.
Then, $Der_{c}(G,(A,\mu))$ is a closed subspace of $Der_{c}(G,A)\times R$;
  \begin{itemize}
    \item[(ii)] Let $A$ be a topological $G$-module. Then, $Der_{c}(G,A)$ is a closed subspace of $\mathcal{C}_{k}(G,A)$.
  \end{itemize}
\label{prop 2.17}\end{prop}
\begin{proof} (i).  Consider the map $$\phi_{g}: \mathcal{C}_{k}(G,A)\times R\rightarrow R, (\alpha,r)\mapsto r^{-1}\mu\alpha(g)^{g}r,$$
for $g\in G$. By 9.6 Lemma of \cite{13}, $\phi_{g}$ is continuous, for all $g\in G$. Hence, $\phi_{g}^{-1}(1)$ is closed in $\mathcal{C}_{k}(G,A)\times R$, for all $g\in G$. It is easy to see that $$Der_{c}(G,(A,\mu))=\bigcap_{g\in G}\phi_{g}^{-1}(1)\bigcap (Der_{c}(G,A)\times R).$$ Therefore, $Der_{c}(G,(A,\mu))$ is closed in $Der_{c}(G,A)\times R$.
\par (ii). By a similar argument as in (i), we consider the continuous map $$\chi_{(g,h)}: \mathcal{C}_{k}(G,A)\rightarrow A, \alpha\mapsto \alpha(gh)^{-1}\alpha(g)^{g}\alpha(h),$$  for $(g,h)\in G\times G$. Since $$Der_{c}(G,A)=\bigcap_{(g,h)\in G\times G}\chi_{(g,h)}^{-1}(1),$$ then $Der_{c}(G,A)$ is closed in $\mathcal{C}_{k}(G,A)$.
\end{proof}
 We  immediately  obtain the following two corollaries.
\begin{cor} Let $(A,\mu)$ be a partially crossed topological $G-R$-bimodule. Then, $Der_{c}(G,(A,\mu))$ is a closed subspace of $\mathcal{C}_{k}(G,A)\times R$.
\label{cor 2.18}\end{cor}
\begin{cor} Let $G$ be a topological group and $A$ an abelian topological group. Then, $Hom_{c}(G,A)$ is a closed subgroup of $\mathcal{C}_{k}(G,A)$.
\label {cor 2.19}\end{cor}
Suppose that   $(A,\mu)$ is a partially crossed topological $G-R$-bimodule. There is an action of $G$ on $Der(G,(A,\mu))$ defined by
$$^{g}(\alpha,r)=(\tilde{\alpha},^{g}r), g\in G, r \in R \eqno{(2.4)}$$
with $\tilde{\alpha}(h)=\ ^{g}\alpha(^{g^{-1}}h), h\in G$ \cite{7}.
\par Note that if $(\alpha,r)\in Der_{c}(G,(A,\mu))$, then $^{g}(\alpha,r)\in Der_{c}(G,(A,\mu)), \forall g\in G$, since $\tilde{\alpha}$ is continuous. This shows that $Der_{c}(G,(A,\mu))$ is a $G$-submodule of $Der(G,(A,\mu))$.
\begin{lem} Let $G$ and $R$ be locally compact groups and $(A,\mu)$ a partially crossed topological module. Then by (2.4), $Der_{c}(G,(A,\mu))$ is a topological $G$-module.
\label{lem 2.20}\end{lem}
\begin{proof} Since $G$ is locally compact and Hausdorff, then the evaluation map $e:G\times \mathcal{C}_{k}(G,A)\rightarrow A$, $(g,\alpha)\mapsto \alpha(g)$ is continuous. Thus, the map $$\Phi:G\times G\times \mathcal{C}_{k}(G,A)\rightarrow A, (g,h,\alpha)\mapsto \ ^{g}\alpha(^{g^{-1}}h)$$ is continuous. By a similar  argument as in the proof of Proposition \ref{prop 2.15},  the map $G\times \mathcal{C}_{k}(G,A)\rightarrow \mathcal{C}_{k}(G,A)$, $(g,\alpha)\mapsto \tilde{\alpha}$ is continuous, where $\tilde{\alpha}(h)=\ ^{g}\alpha(^{g^{-1}}h)$, $ h\in G$. Hence, $$(G\times \mathcal{C}_{k}(G,A))\times R\rightarrow \mathcal{C}_{k}(G,A)\times R$$ $$((g,\alpha),r)\mapsto (\tilde{\alpha},^{g}r)$$ is continuous. Therefore, by restriction of this map to $G\times Der_{c}(G,(A,\mu))$ we get the continuous map $$G\times Der_{c}(G,(A,\mu))\rightarrow Der_{c}(G,(A,\mu))$$ $$ ((g,\alpha),r)\mapsto (\tilde{\alpha},^{g}r)$$ and this completes the proof.
\end{proof}
Let $(A,\mu)$ be a partially crossed topological $G-R$-bimodule. If  $G$  is a topological $R$-module, and the compatibility condition
\begin{center}
$^{(^{r}g)}a=\ ^{rgr^{-1}}a$ and $^{(^{r}g)}s=\ ^{rgr^{-1}}s$; $\forall r,s\in R, g\in G, a\in A$,
\end{center}
holds, then $Der(G,(A,\mu))$ is an $R$-module  via
$$^{r}(\alpha,s)=(\tilde{\alpha},^{r}s) \eqno{(2.5)}$$ where $\tilde{\alpha}(g)=\ ^{r}\alpha(^{r^{-1}}g), g\in G$ \cite{7}.
\par It is easy to see that $Der_{c}(G,(A,\mu))$ is an $R$-submodule of $Der(G,(A,\mu))$.
\begin{lem}  Let $G$ and $R$ be locally compact groups and $(A,\mu)$ a partially crossed topological $G-R$-bimodule. Then by (2.5), $Der_{c}(G,(A,\mu))$  is a topological $R$-module.
\label{lem 2.21}\end{lem}
\begin{proof} This  can be  proved by a similar argument as in  Lemma \ref{lem 2.20}.
\end{proof}
\begin{defn} Let $G$ and $R$ be topological groups acting continuously on each other. These actions are said to be compatible if
\begin{center}
$^{(^{r}g)}s=\ ^{rgr^{-1}}s$ and $^{(^{g}r)}h=\ ^{grg^{-1}}h$; $\forall r,s\in R, g,h\in G.$
\end{center}
 Also, it is said that the topological groups $G$ and $R$ act (continuously) on a topological group $A$ compatibly if
\begin{center}
$^{(^{r}g)}a=\ ^{rgr^{-1}}a$ and $^{(^{g}r)}a=\ ^{grg^{-1}}a$; $\forall r\in R, g\in G, a\in A$.
\end{center}
\label{def 2.22}\end{defn}
\begin{prop} Let $G$ and $R$ be locally compact groups and $(A,\mu)$ a partially crossed topological $G-R$-bimodule. Let the topological groups $G$ and $R$ act continuously on each other and on $A$ compatibly. Then, $(Der_{c}(G,(A,\mu)),\gamma)$ is a precrossed topological $G-R$-bimodule, where $\gamma:Der_{c}(G,(A,\mu))\rightarrow R$, $(\alpha,r)\mapsto r$.
\label{prop 2.23}\end{prop}
\begin{proof} Since $G$ and $R$ are locally compact groups, then by Lemma \ref{lem 2.20} and Lemma \ref{lem 2.21}, $G$ and $R$ act continuously on  $Der_{c}(G,(A,\mu))$. The map $\gamma$ is continuous, since $\pi_{2}:\mathcal{C}_{k}(G,A)\times R\rightarrow R, (\alpha,r)\mapsto r$ is continuous. Also, $\gamma$ is a $G$-homomorphism and an $R$-homomorphism. Since $^{^{g}r}(\alpha,s)=\ ^{grg^{-1}}(\alpha,s)$ for all $g\in G, r\in R, (\alpha,s)\in Der_{c}(G,(A,\mu))$ (Proposition 5 of \cite{7}), we conclude that $(Der_{c}(G,(A,\mu)),\gamma)$ is  a precrossed topological $G-R$-bimodule.
\end{proof}

\section{The first non-abelian cohomology of a topological group as a topological space}\label{section 3}

In this section we define the first non-abelian cohomology $H^{1}(G,(A,\mu))$ of $G$   with coefficients in a partially crossed topological $G-R$-bimodule $(A,\mu)$. We will introduce a topological structure on $H^{1}(G,(A,\mu))$. It will be shown that under what conditions $H^{1}(G,(A,\mu))$ is a topological group. As a result, $H^{1}(G,(A,\mu))$ is a topological group for every partially crossed topological $G$-module. In addition, we verify some topological properties of $H^{1}(G,(A,\mu))$.
 \par Let $R$ ba a topological $G$-module, then we define $$H^{0}(G,R)=\{r|^{g}r=r,\forall g\in G\}.$$
\par Let $(A,\mu)$ be a partially crossed topological $G-R$-bimodule. H. Inassaridze \cite{6} introduced an equivalence relation on the group $Der(G,(A,\mu))$ as follows:
\begin{center}
$(\alpha,r)\sim(\beta,s)\Leftrightarrow$ $(\exists \ a\in A \wedge ( \forall g\in G\Rightarrow\beta(g)=a^{-1}\alpha(g)^{g}a))$ \\ $\wedge \ (s=\mu(a)^{-1}r$ mod $H^{0}(G,R))$
\end{center}
\par Let $\sim'$ be the restriction of $\sim$ to $Der_{c}(G,(A,\mu))$. Therefore, $\sim'$ is an equivalence relation. In other word, $(\alpha,r)\sim'(\beta,s)$ if and only if $(\alpha,r)\sim(\beta,s)$, whenever $(\alpha,r), (\beta,s)\in Der_{c}(G,(A,\mu))$.
\begin{defn} Let $(A,\mu)$ be a partially crossed topological $G-R$-bimodule. The quotient set  $Der_{c}(G,(A,\mu))/\sim'$  will be called the first cohomology  of $G$ with the coefficients in $(A,\mu)$ and is denoted by $H^{1}(G,(A,\mu))$. (In this definition, the groups $G$, $R$ and $A$ are not necessarily Hausdorff.)
\label{def 3.1}\end{defn}
\begin{thm} Let $G$ and $R$ be  locally compact groups and $(A,\mu)$ a partially crossed topological $G-R$-bimodule satisfying
the following conditions
\begin{itemize}
  \item[(i)] $H^{0}(G,R)$ is a normal subgroup of $R$;
  \item[(ii)] for every $c\in H^{0}(G,R)$ and $(\alpha,r)\in Der_{c}(G,(A,\mu))$, there exists
$a \in A$ such that $\mu(a)=1$ and $^{c}\alpha(g)=a^{-1}\alpha(g)^{g}a$, $\forall g\in G$.
\end{itemize}
Then, $Der_{c}(G,(A,\mu))$ induces  a topological group structure on $H^{1}(G,(A,\mu))$.
\label{thm 3.2}\end{thm}
\begin{proof}  By Theorem 2.1 of \cite{6},  the group $Der(G,(A,\mu))$ induces the following action on $Der(G,(A,\mu))/\sim$
 $$[(\alpha,r)][(\beta,s)]=[(^{r}\beta\alpha,rs)].$$
Thus, $N=\{(\alpha,r)|(\alpha,r)\in Der(G,(A,\mu)), (\alpha,r)\sim(\mathbf{1},1)\}$ is a normal subgroup of $Der(G,(A,\mu))$. Therefore, $N'=\{(\alpha,r)|(\alpha,r)\in Der_{c}(G,(A,\mu)), (\alpha,r)\sim(\mathbf{1},1)\}$ is a normal subgroup of $Der_{c}(G,(A,\mu))$. By Theorem \ref{thm 2.16}, $Der_{c}(G,(A,\mu))$ is a topological group. Obviously,  $H^{1}(G,(A,\mu))=Der_{c}(G,(A,\mu))/N'$. Therefore, $H^{1}(G,(A,\mu))$ is a topological group.
\end{proof}
\begin{notice} (i) Note that Hausdorffness of $A$ is not needed in Theorem \ref{thm 3.2}.
\begin{itemize}
\item[(ii)] Let $A$ be a topological $G$-module. The first cohomology, $H^{1}(G,A)$, of $G$ with coefficients in $A$  is defined as in \cite{12}. Thus, the compact-open topology on $Der_{c}(G,A)$ induces a quotient topology on $H^{1}(G,A)$. From now on, we consider $H^{1}(G,A)$ with this topology. Define $Inn(G,A)=\{Inn(a)|a\in A\}$, where for all $a\in A, g\in G$, $Inn(a)(g)=a^{g}a^{-1}$. If $A$ is abelian, then by Remark 2.4. (i) of \cite{12},  $Inn(G,A)$ is a normal subgroup of $Der_{c}(G,A)$ and $H^{1}(G,A)=Der_{c}(G,A)/Inn(G,A)$; moreover, $H^{1}(G,A)$ is a topological group, and it is Hausdorff if and only if $Inn(G,A)$ is closed in $Der_{c}(G,A)$.
\item[(iii)] Define $Inn(G,(A,\mu))=\{(Inn(a),\mu(a)z)|a\in A, z\in H^{0}(G,R)\}$. Note that if, $H^{1}(G,(A,\mu))$ is a topological group, then $Inn(G,(A,\mu))$ is a normal subgroup of $Der_{c}(G,(A,\mu))$. Thus, by hypotheses of Theorem \ref{thm 3.2}, $Inn(G,(A,\mu))$ is a normal subgroup of $Der_{c}(G,(A,\mu))$ and $H^{1}(G,(A,\mu))=Der_{c}(G,(A,\mu))/Inn(G,(A,\mu))$.
\end{itemize}
\label{notice 3.3}\end{notice}
\par In the following, we give an example for this fact that: in general, $H^{1}(G,A)$ and $H^{1}(G,(A,\mu))$ are not necessarily Hausdorff.
\begin{example}  Let $G$ be an  abelian discrete group; let  $(\mathbb{Z},+)$  be the integer numbers group with the indiscrete topology $\tau$, (i.e., $\tau=\{\mathbb{Z},\emptyset\}$) such that $\chi:G\to Aut(\mathbb{Z})\cong\mathbb{Z}/2\mathbb{Z}$ is a nontrivial homomorphism. Equip $Aut(\mathbb{Z})$  with the compact-open topology. Then, $\chi$ induces a nontrivial continuous action of $G$ on $\mathbb{Z}$ given by $^{g}z=\chi(g)(z)$, $\forall g\in G , z\in \mathbb{Z}$.  For all $g\in G$, we have $[\{g\},\mathbb{Z}]\cap Der_{c}(G,\mathbb{Z})=Der_{c}(G,\mathbb{Z})$. Hence, the compact-open  topology on $Der_{c}(G,\mathbb{Z})$ is the indiscrete topology. Thus, $H^{1}(G,\mathbb{Z})=Der_{c}(G,\mathbb{Z})/Inn(G,\mathbb{Z})$ has the indiscrete topology.  On the other hand, discreteness of $G$ implies that $Der_{c}(G,\mathbb{Z})=Der(G,\mathbb{Z})$. Hence by Theorem 3.2 of \cite{1}, $H^{1}(G,\mathbb{Z})\cong\mathbb{Z}/2\mathbb{Z}\neq 1$. Hence, $H^{1}(G,\mathbb{Z})$ is not Hausdorff. Consequently, $Inn(G,\mathbb{Z})$ is not closed in $Der_{c}(G,\mathbb{Z})$. Now, note that $(\mathbb{Z},\mathbf{1}:\mathbb{Z}\to G)$ is a crossed $G-G$-bimodule. It is easy to see that $Inn(G,(\mathbb{Z},\mathbf{1}))=Inn(G,\mathbb{Z})\times G$. Hence $Inn(G,(\mathbb{Z},\mathbf{1}))$ is not closed in $Der_{c}(G,(\mathbb{Z},\mathbf{1}))$ and so $H^{1}(G,(\mathbb{Z},\mathbf{1}))$ is not Hausdorff.
\label{example 3.4}\end{example}
\begin{rem} Let $A$ be an abelian topological {$G$-module and  $A$ be compact} Hausdorff. Then, $H^{1}(G,A)$ is a Hausdorff topological group.
\label{rem 3.5}\end{rem}
\par Let $(A,\mu)$ be a partially crossed $G$-module. Naturally  $(A,\mu)$ is a crossed $G-G$-bimodule. Thus, we define the first cohomology of $G$ with coefficients in $(A,\mu)$ {as} the set $H^{1}(G,(A,\mu))$.
\begin{thm} Let $G$ be a locally compact group and $(A,\mu)$ a partially crossed topological $G$-module. Then, $H^{1}(G,(A,\mu))$ is a topological group. In addition, if any of the following conditions is satisfied, then $H^{1}(G,(A,\mu))$ is Hausdorff.
\begin{itemize}
   \item[(i)] $A$ is compact and $G$ has trivial center;
   \item[(ii)] $A$ is a trivial $G$-module;
   \item[(iii)] $A$ and $Z(G)$ are compact, in particular if both topological groups $A$ and $G$ are compact.
 \end{itemize}
\label{thm 3.6} \end{thm}
\begin{proof} Note that $H^{0}(G,G)=Z(G)$. For any $c\in Z(G)$ and $(\alpha,g)\in Der_{c}(G,(A,\mu))$,{} $\alpha(cx)=\alpha(xc)$ for all $x\in G$. Thus, $^{c}\alpha(x)=\alpha(c)^{-1}\alpha(x)\alpha(c)$, $\forall x\in G$ and $\mu(\alpha(c))=\ ^{g}cc^{-1}=1$. Since $G$ is locally compact, then by Theorem \ref{thm 3.2}, $H^{1}(G,(A,\mu))$ is a topological group.
\par (i). If $A$ is compact and $G$ has trivial {center then by the assumption} $Z(G)=1$.  So $Inn(G,(A,\mu))=\{(Inn(a),\mu(a))|a\in A\}$. It is easy to see that the map $Inn:A\to Der_{c}(G,A), a\mapsto Inn(a)$ is continuous. Thus, compactness of $A$ implies that $Inn(G,(A,\mu))$ is a compact subset of $Der_{c}(G,(A,\mu))$. Hence, $Inn(G,(A,\mu))$ is closed in $Der_{c}(G,(A,\mu))$. So $H^{1}(G,(A,\mu))$ is Hausdorff.
\par (ii). If $G$ acts trivially on $A$, then $^{g}\mu(a)=\mu(a)$, for every $g\in G$ and $a\in A$. Thus, $Inn(G,(A,\mu))=\{\mathbf{1}\}\times Z(G)$. Hence, $Inn(G,(A,\mu))$ is closed in $Der_{c}(G,(A,\mu))$.
\par (iii). Consider the continuous map $A\times Z(G)\to Der_{c}(G,(A,\mu))$, $(a,z)\mapsto (Inn(a),\mu(a)z)$. Consequently, the part (iii) is proved.
\end{proof}
\begin{lem}  Let $G$ be a locally compact group and $A$ an abelian topological group. Then, there is a natural topological isomorphism $$Hom_{c}(G,A)\simeq Hom_{c}(G/\overline{[G,G]},A).$$
\label{lem 3.7}\end{lem}
\begin{proof} Since $G$ is locally compact, then $G/\overline{[G,G]}$ is a locally compact group. Let $\pi:G\rightarrow G/\overline{[G,G]}$ be the natural epimorphism. Then, obviously  $\chi:Hom_{c}(G/\overline{[G,G]},A)\rightarrow Hom_{c}(G,A)$, $f\mapsto \pi f$  is a one to one and onto continuous homomorphism. We show that $\chi$ is an open map. It suffices to show that for every  neighborhood $\Gamma$ of $\mathbf{1}$ in $Hom_{c}(G,A)$, {} $\chi(\Gamma)$ is a neighborhood of $\mathbf{1}$ in $Hom_{c}(G/\overline{[G,G]},A)$. Since $Hom_{c}(G,A)$ is a topological group, so it is a homogeneous space.  It is clear that the network of  all compact subset of $G$ is closed under finite unions. Now, by a similar argument as in page 7 of \cite{9},  there is an open neighborhood of $\mathbf{1}$ of the form $S(C,U)$ in $\Gamma$. Note that $S(C,U)=\{f|f\in Hom_{c}(G/\overline{[G,G]},A), f(C)\subset U\}$, where $C$ is compact in $G/\overline{[G,G]}$ and $U$ is  open in $A$.  Since $G$ is locally compact, then by 5.24.b of \cite{4},  there is a compact subset $D$ of $G$ such that $\pi(D)=C$.  It is easy to see that $\chi(S(C,U))=S(D,U)\subset \chi(\Gamma)$. Therefore, $\chi$ is a topological isomorphism.
\end{proof}
Recall that a topological group $G$ \emph{has no small subgroups} (or \emph{is without small subgroups}) if there is a neighborhood of the identity that contains no nontrivial subgroup of $G$. For example if $n$ is a positive integer number, then the $n$-dimensional vector group, the $n$-dimensional tours, and general linear groups over the complex numbers are without small subgroups. It is well-known that the property of having no small subgroups is an extension property (see 6.15 Theorem of \cite{13}). A topological group $G$ is called compactly generated if there exists a compact
subset $K$  so that it generates $G$, that is $G=<K>$.
\begin{prop} (1) If $G$ is  a locally compact group and $A$ is a compact abelian group without small subgroups, then $Hom_{c}(G,A)$ is a locally compact group.
\begin{itemize}

   \item[(2)] If $G$ is a locally compact compactly generated group  and $A$ is a locally compact abelian group without small subgroups, then $Hom_{c}(G,A)$ is a locally compact group.
  \item[(3)] If $G$ is a compact group  and $A$  is an abelian group without small subgroups, then $Hom_{c}(G,A)$ is a discrete group.
  \item[(4)] If $G$ is a discrete group and $A$ is a compact group, then $Hom_{c}(G,A)$ is a compact group.
  \item[(5)] If $G$ is a finite discrete group and $A$ is a compact abelian group without small subgroups, then $Hom_{c}(G,A)$ is a finite discrete group.
  \item[(6)] Let $A$ be a topological $G$-module. If $G$ is  discrete and $A$ is compact, then $Der_{c}(G,A)$ is a compact group.

\end{itemize}
\label{prop 3.8}\end{prop}
\begin{proof} Since $A$ is abelian, by Lemma \ref{lem 3.7},  $Hom_{c}(G,A)\simeq Hom_{c}(G/\overline{[G,G]},A)$. Therefore, (1) and (2) follow from two corollaries in page 377 of \cite{10}. Also (3) is obtained by  Theorem 4.1 of \cite{10}.
\par (4) Since $G$ is discrete, then $\mathcal{C}_{k}(G,A)=\mathcal{C}_{p}(G,A)$. By Corollary \ref{cor 2.19}, $Hom_{c}(G,A)$ is closed in $\mathcal{C}_{k}(G,A)$. Let $B=\prod_{g\in G}A_{g}$, where $A_{g}=A, \forall g \in G$. It is clear that the map $\Phi:\mathcal{C}_{p}(G,A)\rightarrow B$, $f\mapsto \{f(g)\}_{g\in G}$ is  continuous. In addition, since $G$ is discrete, then the map $G\times B\rightarrow A$ ,$(h,\{a_{g}\}_{g\in G})\mapsto a_{h}$ is continuous. Hence, this map induces the continuous map $\Psi:B\rightarrow \mathcal{C}_{p}(G,A)$, $\{a_{g}\}_{g\in G}\mapsto f$, where $f(g)=a_{g}$. Obviously, $\Phi\Psi=Id$ and $\Psi\Phi=Id$. Consequently, $\mathcal{C}_{p}(G,A)$ is homeomorphic to $B$. Thus, $\mathcal{C}_{p}(G,A)$ is compact. So $Hom_{c}(G,A)$ is compact.
\par (5) This is an immediate result from (3) and (4).
\par (6) By Proposition \ref{prop 2.17}, $Der_{c}(G,A)$ is closed in $\mathcal{C}_{k}(G,A)$. We have seen in the proof of (4) that $\mathcal{C}_{k}(G,A)$ is compact. Consequently, $Der_{c}(G,A)$ is compact.
\end{proof}
 Recall that a topological space $X$ is called a $k$-space if every
subset of $X$, whose intersection with every compact $K\subset X $ is relatively
open in $K$, is open in $X$. A topological space $X$ is a $k$-space if and only if $X$ is the quotient image of a locally compact space (see Characterization (1) of \cite{14}). For example, locally compact spaces and first-countable spaces are $k$-spaces. It is well-known that the $k$-space  property is preserved by the closed {subsets and the quotients}. Also, the product of a locally compact space with a $k$-space is a $k$-space (see Result (1) of \cite{14}). We call a topological group {to be} a $k$-group if it is a $k$-space as a topological space.
\begin{thm} Let $G$ be a locally compact group; let $(A,\mu)$ be a partially crossed topological $G-R$-bimodule such that $G$ acts trivially on $A$ and $R$.
\begin{itemize}
  \item[(1)] If $R$ is a  $k$-group and $A$ is compact  without small subgroups, then $H^{1}(G,(A,\mu))$ is a $k$-space.

  \item[(2)] If $G$ is compactly generated, $R$ is a $k$-group and $A$ is locally compact  without small subgroups, then $H^{1}(G,(A,\mu))$ is a $k$-space.
  \item[(3)] If $G$  is  compact, $A$  has no small subgroups and $R$ is discrete, then $H^{1}(G,(A,\mu))$ is discrete.
  \item[(4)] If $G$  and $R$ are finite discrete and $A$ is compact without small subgroups, then $H^{1}(G,(A,\mu))$ is a finite discrete space.

\end{itemize}
\label{thm 2.9}\end{thm}
\begin{proof} Since $G$ acts trivially on $A$ and $R$, then it is easy to see that $Der_{c}(G,(A,\mu))$ is homeomorphic to  $Hom_{c}(G,Ker\mu)\times R$. Note that $Ker\mu$ is closed in $Z(A)$. Now by  Proposition \ref{prop 3.8}, the assertions (1) to (4) hold.
\end{proof}
\begin{thm} Let $G$ be a locally compact abelian topological group; let $(A,\mu)$ be a partially crossed topological $G$-module and $A$ a trivial $G$-module.
\begin{itemize}
  \item[(1)] If $A$ is compact  without small subgroups, then $H^{1}(G,(A,\mu))$ is a locally compact abelian group.

  \item[(2)] If $G$ {is} compactly generated  and $A$ is locally compact  without small subgroups, then $H^{1}(G,(A,\mu))$ is a locally compact abelian group.
  \item[(3)] If $G$ is finite discrete  and $A$ is compact without small subgroups, then $H^{1}(G,(A,\mu))$ is a finite discrete abelian group.
\end{itemize}
\label{thm 2.10}\end{thm}
\begin{proof} Since $G$ is a locally compact abelian group and acts trivially on $A$, one can see $Der_{c}(G,(A,\mu)) \simeq Hom_{c}(G,Ker\mu)\times G$. Therefore, by Proposition \ref{prop 3.8}, the proof is completed.
\end{proof}
Let $G$ and $A$ be topological groups; let $K$ be an abelian subgroup of $A$. We denote the set of all continuous homomorphisms $f:G\to A$ with $f(G)\subset K$ by $Hom_{c}(G,A|K)$. Obviously, if $G$ is locally compact, then $Hom_{c}(G,A|K)$ with compact-open topology is an abelian topological group.
\begin{rem} (1) Let $(A,\mu)$ be a partially crossed topological $G$-module. Suppose that $G$ is a locally compact abelian group {which acts} trivially on $A$. Then, $H^{1}(G,(A,\mu))\simeq Hom_{c}(G,A|Ker\mu)$.
\begin{itemize}  \item[(2)] Let $A$ be an  abelian topological $G$-module. Then, $(A,\mathbf{1})$ is a crossed topological $G-R$-bimodule for every topological group $R$, and $H^{1}(G,(A,\mathbf{1}))$ is homeomorphic to $H^{1}(G,A)$.
\item[(3)] Let $G$ be a locally compact group and $A$ an abelian topological $G$-module. Then,
 $(A,\mathbf{1})$ is a crossed topological $G$-module, and $H^{1}(G,(A,\mathbf{1})) \simeq H^{1}(G,A)$. In particular if $G$ acts trivially on $A$, then $H^{1}(G,(A,\mathbf{1})) \simeq Hom_{c}(G/\overline{[G,G]},A)$.
\item[(4)] Let $G$ be a locally compact group and $A$ an abelian topological $G$-module. Then,
     $H^{1}(G,(A,\pi_{A}))=H^{1}(G,(A,\mathbf{1}))\simeq H^{1}(G,A)$.
\end{itemize}
\label{rem 3.11}\end{rem}
\begin{thm} Let $(A,\mu)$ be a partially crossed topological $G-R$-bimodule. Suppose that $G$ is a discrete group, $A$ and $R$ are compact. Then, $H^{1}(G,(A,\mu))$ is compact.
\label{thm 3.12}\end{thm}
\begin{proof} By Proposition \ref{prop 2.17}, $Der_{c}(G,(A,\mu))$ is closed in $Der_{c}(G,A)\times R$. Obviously, if $R$ is compact, then $H^{1}(G,(A,\mu))$ is compact.
\end{proof}
 As an immediate result of Theorem \ref{thm 3.12}, we have the following corollary:
 \begin{cor}Let $(A,\mu)$ be a partially crossed topological $G$-module, $G$ {be} finite discrete and $A$ {be compact. Then} $H^{1}(G,(A,\mu))$ is a compact group.
 \label{cor 3.13}\end{cor}
\begin{defn} A topological group $A$ is radical-based, if it has a countable base $\{U_{n}\}_{n\in \mathbb{N}}$
at 1, such that each $U_{n}$ is symmetric and for all $n\in \mathbb{N}$:
\begin{itemize}
 \item[(1)] $(U_{n})^{n} \subset U_{1}$;
 \item[(2)] $a, a^{2},...,a^{n}\in U_{1}$ implies $a\in U_{n}$.
 \end{itemize}
\label{def 3.14}\end{defn}
For example, if $n$ is a positive integer, then the $n$-dimensional vector group, the $n$-dimensional torus  and the rational numbers are radical-based groups. For another example see \cite{8}.
\begin{thm} Let $(A,\mu)$ be a partially crossed topological $G-R$-bimodule, and $G$ a first countable group. Let  $R$ be locally compact and $A$ a compact radical-based group with $H^{0}(G,A)=A$. Then, $H^{1}(G,(A,\mu))$ is a $k$-space.
\label{thm 3.15}\end{thm}
\begin{proof} Since $H^{0}(G,A)=A$, then it follows from Proposition \ref{prop 2.17} that $Der_{c}(G,(A,\mu))$ is closed in $Hom_{c}(G,A)\times R$. By Theorem 1 of \cite{8}, $Hom_{c}(G,A)$ is a $k$-space. Thus, $Hom_{c}(G,A)\times R$ is a $k$-space. Consequently, $H^{1}(G,(A,\mu))$ is a $k$-space.
\end{proof}
 By Theorem \ref{thm 3.15}, the next corollary is immediate.
\begin{cor} Let $(A,\mu)$ be a partially crossed topological $G$-module, let $G$ be locally compact first countable and $A$ a compact radical-based group with $H^{0}(G,A)=A$. Then, $H^{1}(G,(A,\mu))$ is a $k$-group.
\label{cor 3.16}\end{cor}

\noindent \Large\textbf{Acknowledgment}\\[2mm]
\footnotesize The authors gratefully acknowledge the financial support for this work that was provided by University of Guilan.\\[3mm]

\end{document}